\documentclass[12pt,epsfig]{article}

\usepackage{amsmath, amssymb, amsthm}
\usepackage[margin=1.2 in]{geometry}

\usepackage[pdftex]{graphicx}

\numberwithin{equation}{section}
\newtheorem{thm}{Theorem}[section]
\newtheorem{lem}{Lemma}[section]
\newtheorem{rem}{Remark}[section]
\newtheorem*{rem*}{Remark}
\newtheorem{prop}{Proposition}[section]

\title 
{Negative time splitting is stable}
\date{}

\author
{Dong Li
\thanks
{Department of Mathematics, the Hong Kong University of Science
\& Technology, Clear Water Bay, Hong Kong.
 Email: {mpdongli@gmail.com}.
 }\qquad
{Chaoyu Quan}	
\thanks{SUSTech International Center for Mathematics, Southern University of Science and Technology,	Shenzhen, China.
Email: {quancy@sustech.edu.cn}.
}
}

\begin{document}
\maketitle
\begin{abstract}
For high order (than two) in time operator-splitting methods applied to dissipative systems, a folklore
issue is the appearance of negative-time/backward-in-time linear evolution operators such as backward heat operators  interwoven with nonlinear evolutions.  The  stability of such methods has remained an
ensuing difficult open problem. In this work we consider a  fourth order operator
splitting discretization for the Allen-Cahn equation which is a prototypical high order
splitting method with negative time-stepping, i.e. backward in time integration for the linear
parabolic part.   We introduce a new theoretical framework
and prove uniform energy stability and  higher
Sobolev stability. This is the first strong stability result for negative time stepping operator-splitting methods.
\end{abstract}
\section{Introduction}
We consider  the  Allen-Cahn equation
\begin{align} \label{1}
\begin{cases}
\partial_t u  =  \nu \Delta u -f (u) , \quad (t,x) \in (0, \infty) \times \Omega, \\
u \Bigr|_{t=0} =u_0,
\end{cases}
\end{align}
where  the  unknown $u=u(t,x):\, [0,\infty)\times \Omega \to \mathbb R$. The parameter $\nu>0$ is called the mobility coefficient and we fix it  as a constant for simplicity.  The nonlinear term takes
the form $f(u)=u^3-u =F^{\prime}(u)$, where $F(u) = \frac 14 (u^2-1)^2$ is the standard double well. To minimize technicality,  we  take the spatial domain $\Omega$ in \eqref{1} as the $2\pi$-periodic torus $\mathbb T=\mathbb R/ 2\pi \mathbb Z
=[-\pi,\pi]$.  With some additional work our analysis can be extended to physical dimensions $d=2, 3$.   The system \eqref{1}  arises as a $L^2$-gradient flow of a Ginzburg-Landau type energy
functional $ E(u)$, where
\begin{equation}
 E(u)= \int_{\Omega} \left( \frac 12 \nu |\nabla u|^2 + F(u) \right) dx
=\int_{\Omega} \left( \frac 12 \nu |\nabla u|^2 + \frac 14 (u^2-1)^2 \right) dx.
\end{equation}
The basic energy conservation law takes the form
\begin{equation}
\frac d {dt}  E ( u(t) ) +
\|   \partial_t u \|_2^2
=\frac d {dt}  E(u(t)) + \int_{\Omega}
|  \nu \Delta u - f(u )  |^2 dx =0.
\end{equation}
It follows that 
\begin{align}
E(u(t_2) ) \le E(u(t_1) ), \qquad \forall\, 0\le t_1<t_2.
\end{align}
Besides the $L^2$-type conservation law, there is also $L^{\infty}$-type control. 
Due to special form of the nonlinearity, for smooth solutions we have the maximum principle
\begin{align}
\| u(t,\cdot) \|_{L_x^{\infty}} \le \max\{1, \, \| u_0 \|_{L_x^{\infty}} \}, \qquad\forall\, t\ge 0.
\end{align}
 As a consequence the long-time wellposedness and regularity is not an issue for \eqref{1}. 

The objective of this work is to establish strong stability of a fourth order in time operating
splitting algorithm applied to the Allen-Cahn equation \eqref{1}. This is a part of our on-going
program to develop a new theory for the rigorous analysis of stability and convergence of
operator-splitting methods applied to dissipative-type problems. Due to various subtle  technical obstructions, there were very few rigorous results on the analysis of the operator-splitting
type algorithms for the Allen-Cahn equation,  the Cahn-Hilliard equation and similar models. 
Prior to our recent series of works \cite{LQ1a, LQ1b}, most existing results in the literature are conditional one way or another.  To put things into perspective, we briefly review a few closely related representative works and more recent developments.

\begin{itemize}
\item \underline{The work of Gidey-Reddy}. 
Gidey and Reddy  considered in \cite{Red19} a convective Cahn-Hilliard
model of the form
\begin{align} \label{1.14}
\partial_t u - \gamma \nabla \cdot \mathbf{h}(u) + \epsilon^2 \Delta^2 u
=\Delta (f(u)),
\end{align}
where $\mathbf{h}(u) =\frac 12 (u^2, u^2)$.  They adopted an operator-splitting of \eqref{1.14}  into the hyperbolic part, nonlinear diffusion part and diffusion part respectively. Several conditional results  concerning certain weak solutions were obtained. 

\item \underline{The work of Weng-Zhai-Feng}.  In \cite{Feng19},
Weng, Zhai and Feng studied a viscous Cahn-Hilliard model:
\begin{align}
(1-\alpha) \partial_t u = \Delta ( - \epsilon^2 \Delta u + f(u ) + \alpha \partial_t u),
\end{align}
where the parameter $\alpha \in (0, 1)$. The authors employed a fast explicit Strang-type
operator splitting and proved the  stability and the convergence (see Theorem 1 on pp. 7 of \cite{Feng19}) under the assumption that $A=\|\nabla u^{\operatorname{num}}\|^2_{\infty}$, $B=
\| u^{\operatorname{num}} \|_{\infty}^2$ stay bounded,  and satisfy a technical condition
$6A+8-24B>0$. Here $u^{\mathrm{num}}$ denotes the numerical solution.

\item \underline{The work of Cheng-Kurganov-Qu-Tang}.
In \cite{Tang15}, Cheng, Kurganov, Qu and Tang considered the Cahn-Hilliard equation 
\begin{align}
\partial_t u = -\nu \Delta^2 u -\Delta u + \Delta (u^3)
\end{align}
and the MBE equation
\begin{align}
\partial_t \phi = -\delta \Delta^2 \phi - \nabla \cdot ( (1-|\nabla \phi|^2) \nabla \phi ). 
\end{align}
Concerning the Cahn-Hilliard equation, the authors considered a Strang-type splitting approximation
of the form
\begin{align}
u(t+\tau) \approx S_L^{(1)}(\frac {\tau} 2) S_N^{(1)}(\tau) S_L^{(1)}(\frac {\tau}2)  u(t),
\end{align}
where
\begin{align}
&S_L^{(1)}(\frac {\tau}2) = \exp( \frac 12\tau(-\nu \Delta^2 -\Delta) ); 
\end{align}
and $w=S_N^{(1)}(\tau) a$ is the nonlinear propagator 
\begin{align}
\begin{cases}
\partial_t w = \Delta(w^3), \\
w\Bigr|_{t=0} = a.
\end{cases}
\end{align}
Various conditional results were given in \cite{Tang15} but the rigorous analysis of energy stability 
was a long-standing open problem. This problem and several related open problems were 
settled in our recent proof \cite{LQ1b}.
\end{itemize}

In recent \cite{LQ1a},  we carried out the first  energy-stability analysis of a first order operator-splitting approximation  of the Cahn-Hilliard equation.  More precisely denote $u^{\mathrm{CH}}$
as the exact PDE solution to the Cahn-Hilliard equation $\partial_t u 
=-\nu \Delta^2 u -\Delta u + \Delta(u^3)$.  We considered a splitting approximation of the form:
\begin{align}
u^{\mathrm{CH}}(t+\tau) \approx S_L^{(2)}(\tau) S_N^{(2)}(\tau)  u^{\mathrm{CH}}(t),
\end{align}
where $S_L^{(2)}(\tau)=\exp(-\tau \nu\Delta^2  )$ and  $w= S_N^{(2)}(\tau) a$ solves
\begin{align}
\frac {w - a} {\tau} =   \Delta ( a^3 -a).
\end{align}
We introduced a novel modified energy and rigorously proved monotonic decay of the new modified
energy which is coercive in $H^1$-sense.  Moreover we obtained uniform control of higher Sobolev regularity and rigorously justified the first order convergence of the method on any finite
time interval. 

In \cite{LQ1b}, we settled the difficult open problem of energy stability of the Strang-type
algorithm applied to the Cahn-Hilliard equation which was introduced in the work of 
Cheng, Kurganov, Qu and Tang \cite{Tang15}.  One should note that the new theoretical framework
developed in \cite{LQ1b} is completely different from the first order case \cite{LQ1a}. 
In the second-order case, for most generic data 
one no longer has strict energy-monotonicity at disposal and several new ideas
such as dichotomy analysis, an absorbing-set approach were introduced in \cite{LQ1b} in order
to settle long time unconditional (i.e. independent of time-step or time interval) Sobolev bounds
of the numerical solution. 

In this work we develop further the program initiated in \cite{LQ1a, LQ1b} and turn to the analysis of
 higher order in time splitting methods.  A folklore issue is that operator-splitting methods with order
higher than two require negative time-stepping, i.e. backward time integration for each
time step. In \cite{Sh89}, Sheng considered dimensional splitting for the two-dimensional parabolic 
problem $\partial_t u = a \partial_{xx} u  + b \partial_{yy} u$. Using semi-discretization one obtains
the ODE system $\partial_t \mathbf{u}= A \mathbf{u} + B \mathbf{u}$, where the matrices $A$ and
$B$ correspond to the discretization of $a\partial_{xx}$ and $b\partial_{yy}$ respectively. One is then
naturally led to the approximation of $\mathbf{u} =e^{t(A+B)} \mathbf{u}_0 \approx
e^{\frac 12 t A} e^{t B} e^{\frac 12 t A}\mathbf{u}_0$. In \cite{Sh89}, Sheng showed that if
the non-commuting matrices $A$ and $B$ have eigen-values in the left half of the complex-plane,
and one employs the approximation
\begin{align}
&e^{t(A+B)} \approx \sum_{k=1}^K \gamma_k e^{\alpha_{k,1} tA}
e^{\beta_{k,2} t B} \cdots e^{\alpha_{k,m(k)} t A} e^{\beta_{k,m(k)} t B}, \\
&\gamma_k>0, \; \alpha_{i,k}\ge 0, \; \beta_{l,k}\ge 0, 
\end{align}
then the highest order of a stable approximation is two even if $K$ is chosen to be large. Put it 
differently, Sheng's fundamental result states that for an $N^{\mathrm{th}}$-order ($N\ge 3$)
partitioned  split-step schemes, at least one of the solution operators must be applied with
negative time step, i.e. backwardly.  In \cite{Su91} Suzuki adopted an elegant time-ordering
principle from quantum mechanics and proved a general nonexistence theorem of positive decomposition for high order splitting approximations. In \cite{GK96}, Goldman and Kaper
strengthened these results further and showed that 
even with partitioned schemes, each solution operator within a convex partition must be 
performed with at least one negative/backward fractional time step.

For deterministic Hamiltonian-type systems, backward time stepping in general does not create instabilities and high-order operating splitting methods have shown promising
effectiveness \cite{TCN09}. On the other hand, due to the negative time stepping,
there are some arguments that higher-than-three operator splitting methods
cannot be applied to parabolic equations with diffusive terms (\cite{Sh89, Ch04, Book16}).
As we shall explain momentarily, these concerns are not unsubstantiated and even turn
up in the well-defined-ness of these algorithms.

Although rigorous analysis of these issues were not available before, recently Cervi and Spiteri
\cite{CSp18} considered three third-order operating-splitting methods and demonstrated
via extensive numerical simulations the effectiveness of higher order splitting
methods for representative cardiac electrophysiology simulations.  These compelling numerical
evidences propel us to re-examine in detail the stability property of general negative/backward
time-stepping methods in parabolic problems.  Indeed, the very purpose of this work is to 
break these aforementioned technical barriers and establish a new stability theory for these
problems. 

To set the stage and minimize technicality,  we consider the Allen-Cahn equation \eqref{1} posed on the $2\pi$-periodic torus $\mathbb T=[-\pi, \pi]$. To build some intuition, let $\tau>0$ and consider
\begin{align}
S_L(\tau) = \exp( \nu \tau \partial_{xx} ),
\end{align}
and let $w=S_N(\tau) a$ solve the equation
\begin{align}
\begin{cases}
\partial_t w = w-w^3; \\
w\Bigr|_{t=0} = a.
\end{cases}
\end{align}
An explicit formula for $S_N(\tau)$ is readily available, indeed it is not difficult to check 
that 
\begin{align}
S_N(\tau) a = \frac { e^{2\tau} a} { \sqrt{ 1+ (e^{2\tau}-1) a^2} }.
\end{align}
Define a second-order Strang-type propagator
\begin{align}
S_{2}^{(o)}(\tau) = S_L(\frac {\tau}2) S_N(\tau) S_L(\frac{\tau}2).
\end{align}
Following Yoshida \cite{Yos90},  we consider a 4th order integrator obtained by a symmetric
repetition (product) of the 2nd order integrator:
\begin{align}
S_4^{(o)}(\tau) = S_2^{(o)}(x_1\tau) S_2^{(o)}( -x_0\tau)  S_2^{(o)} ( x_1 \tau),
\end{align}
where
\begin{align}
x_0= \frac {2^{\frac 13}} {2-2^{\frac 13}}, \quad x_1 = \frac 1 {2-2^{\frac 13}}.
\end{align}
Somewhat surprisingly, we first show that the propagator $S_4^{(o)}(\tau)$ is ill-defined
if one does not introduce a judiciously chosen spectral cut-off.

\begin{prop}[Ill-definedness of $S_4^{(o)}(\tau)$ with a spectral cut-off] \label{prop1}
The propagator $S_4^{(o)}(\tau)$ is ill-defined in general. 
\end{prop}

The proof of Proposition \ref{prop1} is given in Section 2. Armed with this important observation,
we are led to introduce a spectral cut-off condition for the propagators.  Let $M\ge 2$ be an integer
and we shall consider the projection operator $\Pi_M$ defined for $f:\mathbb T \to \mathbb R$
via the relation
\begin{align}
\Pi_M f = \frac 1 {2\pi} \sum_{|k| \le M} \widehat{f}(k) e^{i kx},
\end{align}
where $\widehat{f}(k)$ denotes the Fourier coefficient of $f$. We introduce the following
spectral cut-off condition.

\vspace{0.2cm}

\noindent
\textbf{Definition} (Spectral condition).   Let $M\ge 2$ be the spectral truncation parameter.
We shall say $\tau>0$ satisfy the spectral condition if $\tau \le l_0 M^{-2}$ for
some $l_0>0$. 

\begin{rem}
This constraint in $\tau$ is reminiscent of the CFL condition in hyperbolic problems.
Here in the parabolic setting we require that the operator $\tau \partial_{xx}$ to remain
bounded when restricted to the spectral cut-off $|k| \le M$.
\end{rem}

Let $M\ge 2$. We now consider the following modified propagators:
\begin{align}
& S^{(2)}(\tau) = \Pi_M S_L(\frac {\tau}2) S_N(\tau) \Pi_M S_L(\frac {\tau}2); \notag \\
& S^{(4)}(\tau) = S^{(2)} (x_1 \tau) S^{(2)} (-x_0 \tau) S^{(2)} (x_1 \tau).
\end{align}

\begin{thm}[Stability of negative-time splitting methods]\label{thm0}
Let $\nu>0$, $M\ge 2$ and consider the Allen-Cahn equation \eqref{1} on the one-dimensional 
$2\pi$-periodic torus $\mathbb T =[-\pi, \pi]$. Assume the spectral cut-off condition
$\tau \le l_0 M^{-2}$ for some $l_0>0$.  Assume the initial data $u^0=\Pi_M a
\in H^{k_0}(\mathbb T)$ ($k_0\ge 1$ is an integer).  Let $\tau>0$ and define
\begin{align}
u^{n+1} = S^{(4)} u^n, \quad n\ge 0.
\end{align}
There exists a constant $\tau_*>0$ depending only on $l_0$, $\|a\|_{H^1}$ and
$\nu$, such that if $0<\tau \le \tau_*$, then
\begin{align}
\sup_{n\ge 0} \| u^n \|_{H^{k_0}} \le A_1<\infty,
\end{align}
where $A_1>0$ depends on ($\| u^0\|_{H^{k_0} }$, $\nu$, $l_0$, $k_0$). 
\end{thm}
\begin{rem}
One can also show the fourth-order convergence of the operator splitting
approximation. Namely if $u^0\in H^{80}(\mathbb T)$ and let $u$ be
the exact PDE solution corresponding to initial data $a$. Let $0<\tau<\tau_*$
bet the same as in Theorem \ref{thm0}. Then for any $T>0$, we have
\begin{align}
\sup_{n\ge 1, n\tau \le T}  \| u^n - u(n\tau, \cdot ) \|_{L^2(\mathbb T)}
\le C (\tau^4 + M^{-10}),
\end{align}
where $C>0$ depends on ($\nu$, $l_0$, $\|u^0\|_{H^{80}}$, $T$).
The regularity assumption on initial data can be lowered. We shall refrain
from stating such a pedestrian result here and leave its justification
to interested readers as exercises.
\end{rem}

The rest of this paper is organized as follows. In Section $2$ we set up the notation and collect
some preliminary lemmas.  In Section $3$ we give the proof of Theorem \ref{thm0}.
\section{Notation and preliminaries}

For any two positive quantities $X$ and $Y$, we shall write $X\lesssim Y$ or $Y\gtrsim X$ if
$X \le  CY$ for some  constant $C>0$ whose precise value is unimportant.
We shall write $X\sim Y$ if both $X\lesssim Y$ and $Y\lesssim X$ hold.
We write $X\lesssim_{\alpha}Y$ if the
constant $C$ depends on some parameter $\alpha$.
We shall
write $X=O(Y)$ if $|X| \lesssim Y$ and $X=O_{\alpha}(Y)$ if $|X| \lesssim_{\alpha} Y$.

We shall denote $X\ll Y$ if
$X \le c Y$ for some sufficiently small constant $c$. The smallness of the constant $c$ is
usually clear from the context. The notation $X\gg Y$ is similarly defined. Note that
our use of $\ll$ and $\gg$ here is \emph{different} from the usual Vinogradov notation
in number theory or asymptotic analysis.

For any $x=(x_1,\cdots, x_d) \in \mathbb R^d$, we denote $|x| =|x|_2=\sqrt{x_1^2+\cdots+x_d^2}$, and
$|x|_{\infty} =\max_{1\le j \le d} |x_j|$.
Also occasionally we use the Japanese bracket notation:
$\langle x \rangle =(1+|x|^2)^{\frac 12}.$

We denote by $\mathbb T^d=[-\pi, \pi]^d = \mathbb R^d/2\pi \mathbb Z^d$ the usual
$2\pi$-periodic torus.
For $1\le p \le \infty$ and any function $f:\, x\in \mathbb T^d \to \mathbb R$, we denote
the Lebesgue $L^p$-norm of $f$ as
\begin{align*}
\|f \|_{L^p_x(\mathbb T^d)} =\|f\|_{L^p(\mathbb T^d)} =\| f \|_p.
\end{align*}
If $(a_j)_{j \in I}$ is a sequence of complex numbers
and $I$ is the index set, we denote the discrete $l^p$-norm
as
\begin{equation}
\| (a_j) \|_{l_j^p(j\in I)} = \| (a_j) \|_{l^p(I)} =
\begin{cases}
 {\displaystyle \left(\sum_{j\in I} |a_j|^p\right)^{\frac 1p}}, \quad 0<p<\infty, \\
 \sup_{j\in I} |a_j|, \quad \qquad p=\infty.
 \end{cases}
 \end{equation}
 For example,
$ \| \hat f(k) \|_{l_k^2(\mathbb Z^d)} = \left(\sum_{k \in \mathbb Z^d} |\hat f(k)|^2\right)^{\frac 12}$.
If $f=(f_1,\cdots,f_m)$ is a vector-valued function, we denote
$|f| =\sqrt{\sum_{j=1}^m |f_j|^2}$, and
$\| f\|_p = \| ({\sum_{j=1}^m f_j^2})^{\frac 12} \|_p$.
We use similar convention for the corresponding discrete $l^p$ norms for the vector-valued
case.


We use the following convention for the Fourier transform pair:
\begin{equation} \label{eqFt2}
\hat f(k) = \int_{\mathbb T^d} f(x) e^{- i k\cdot x} dx, \quad
 f(x) =\frac 1 {(2\pi)^d}\sum_{k\in \mathbb Z^d} \hat f(k) e^{ ik \cdot x},
\end{equation}
and denote for $0\le s \in \mathbb R$,
\begin{subequations}
\begin{align}
&\|f \|_{\dot H^s} = \|f \|_{\dot H^s(\mathbb T^d)} = \| |\nabla|^s f \|_{L^2(\mathbb T^d)}
\sim \|  |k|^s \hat f (k) \|_{l^2_k (\mathbb Z^d)}, \\
& \| f \|_{H^s} = \sqrt{\| f \|_2^2 + \| f\|_{\dot H^s}^2}  \sim \| \langle
 |k| \rangle^s \hat f(k) \|_{l^2_k(\mathbb Z^d)}.
\end{align}
\end{subequations}

\begin{proof}[Proof of Proposition \ref{prop1}]
To prove Proposition \ref{prop1}, it suffices for us to examine the following statement.
Consider $u= e^{\partial_{xx}} \delta_0$ where $\delta_0$ is the periodic
Dirac comb on $\mathbb T$. Let $0<\tau \ll 1$ and consider
\begin{align}
u= w ( 1+ \tau w^2)^{-\frac 12}.
\end{align}
\underline{Claim}: $e^{-\epsilon_0 \partial_{xx}} w \notin L^2$ for any $\epsilon_0>0$. 

\underline{Proof of Claim}. 
Observe that the Fourier coefficients of $u$ are all positive and
\begin{align}
\widehat{u}(k) \gtrsim \tau^m  \widehat{w^{2m+1}}(k).
\end{align}
Clearly the desired conclusion follows.
\end{proof}

\section{Proof of Theorem \ref{thm0} }

\begin{lem}[One-step $H^k$ stability] \label{lemI5.1}
Let $\nu>0$, $M\ge 2$ and assume the spectral condition $\tau M^{-2} \le l_0$ for some constant $l_0>0$. 
Suppose $a \in H^k(\mathbb T)$, $k\ge 1$ and $\|a\|_{H^k} \le A_0$. Let $0\le y_1 \le y_2<\infty$. There
exists $\tau_1=\tau_1(\nu, k,  y_1,y_2, A_0, l_0)>0$ sufficiently small such that
if $0<\tau \le \tau_1$, then
\begin{align}
& \|\Pi_M S_L(-y_1 \tau) S_N(\tau) S_L(y_2 \tau) a \|_{H^k} \le e^{c_1 \tau} \| a \|_{H^k}, 
\notag \\
& \|\Pi_M S_L(-y_1 \tau) S_N(-\tau) S_L(y_2 \tau) a \|_{H^k} \le e^{c_1 \tau} \| a \|_{H^k},
\notag \\
& \|\Pi_M S_L(y_2 \tau) S_N(\tau)  \Pi_M S_L(-y_1 \tau) a \|_{H^k} \le e^{c_1 \tau} \| a \|_{H^k},
\notag \\
& \|\Pi_M S_L(y_2 \tau) S_N(-\tau)  \Pi_M S_L(-y_1 \tau) a \|_{H^k} \le e^{c_1 \tau} \| a \|_{H^k},
\end{align}
where $c_1>0$ depends on ($\nu$, $k$, $y_1$, $y_2$, $A_0$, $l_0$).
\end{lem}

\begin{rem} \label{remI5.1}
Define
\begin{align}
&T_1 a= \Pi_M S_L(-y_1 \tau) S_N(\tau) S_L(y_2 \tau); \\
&T_2 a= \Pi_M S_L(-y_1 \tau) S_N(-\tau) S_L(y_2 \tau) a; \\
&T_3 a= \Pi_M S_L(y_2 \tau) S_N(\tau)  \Pi_M S_L(-y_1 \tau) a; \\
&T_4 a =\Pi_M S_L(y_2 \tau) S_N(-\tau)  \Pi_M S_L(-y_1 \tau) a.
\end{align}
Later we shall apply Lemma \ref{lemI5.1} $n$-times. In particular we need to bound the expression
\begin{align}
     \| \underbrace{T_{i_1} \cdots T_{i_n}}_{\text{$n$ times}} a \|_{H^k}.
\end{align}     
Since the constant $c_1$ depends on 
the $H^k$-norm of the iterates, it is of importance to give uniform control of the $H^k$-norm.
To resolve this issue, we first choose $A_0$, $\tau_1$ and $c_1$ such that
\begin{align}
\| a\|_{H^k} \le \frac 1{10} A_0, \quad \tau_1=\tau_1(\nu, k, y_1,y_2,k, A_0, l_0),
\quad c_1=c_1(\nu, k, y_1,y_2,A_0, l_0).
\end{align}
We choose $n$ such that
\begin{align}
n \tau \le \frac 1 {c_1}.
\end{align}
Clearly in the course of iteration, the $H^k$-norm of the iterates never exceeds $A_0$.  In particular
\begin{align}
\| T_{i_1} \cdots T_{i_n} a \|_{H^k} \le e^{c_1 n\tau} \|a\|_{H^k} \le e \cdot \frac 1{10} A_0
\le A_0.
\end{align}
\end{rem}

\begin{proof}[Proof of Lemma \ref{lemI5.1}]
We begin by noting that for the ODE
\begin{align}
\partial_t w = \pm w \pm w^3; 
\end{align}
we have the estimate (note that $H^k(\mathbb T)$ is an algebra for $k\ge 1$)
\begin{align}
\frac d{dt} \|w \|_{H^k} \lesssim \| w\|_{H^k} + \| w\|_{H^k}^3.
\end{align}
It follows that for $\tau>0$ sufficiently small,
\begin{align}
&\sup_{0\le t\le \tau} \| S_N(\pm t) b \|_{H^k} \le 2 \| b \|_{H^k}; \\
&\| S_N(\tau) b - b \|_{H^k} \le \tau \cdot O(\| b\|_{H^k} + \| b\|_{H^k}^3).
\end{align}
The desired estimates then easily follows from the above using the spectral condition.
\end{proof}

\begin{lem}[Multi-step $H^1$ stability and higher regularity] \label{lemI5.2}
Let $\nu>0$, $M\ge 2$ and assume the spectral condition $\tau M^{-2} \le l_0$ for some constant $l_0>0$. 
Suppose $a \in H^1(\mathbb T)$ and $\|a\|_{H^1} \le B_1$ for some constant $B_1>0$. 
Define 
\begin{align}
& S^{(2)}(\tau) = \Pi_M S_L(\frac {\tau}2) S_N(\tau) \Pi_M S_L(\frac {\tau}2); \notag \\
& S^{(4)}(\tau) = S^{(2)} (x_1 \tau) S^{(2)} (-x_0 \tau) S^{(2)} (x_1 \tau),
\end{align}
where $x_0 = \frac {2^{\frac 13}} {2-2^{\frac 13}} \approx 1.7$,
$x_1= \frac 1 {2-2^{\frac 13}} \approx 1.35$.  For $n\ge 1$, define
\begin{align}
u^{n} = S^{(4)} (\tau) u^{n-1},
\end{align}
where $u^0=a$.  There
exist $c_1=c_1(\nu, l_0,  B_1)>0$  and
$\tau_2=\tau_2(\nu, l_0, B_1)>0$ such that if $0<\tau \le \tau_2$ (we may assume $c_1\tau_2\le 0.01$), then
\begin{align}
&\| u^n\|_{H^1} \le e^{c_1 \tau} \| u^{n-1} \|_{H^1}, \quad\, 1\le n \le \frac 1 {c_1\tau}; 
\notag \\
& \sup_{1\le n \le \frac 1{c_1\tau} } \| u^n \|_{H^1} \le 3B_1; \notag \\
& \sup_{\frac 1{10c_1\tau} \le n \le \frac 1 {c_1\tau} }
\| u^n \|_{H^{80}} \le B_2, \label{I5.16}
\end{align}
where $B_2>0$ depends on ($\nu$, $l_0$, $B_1$). 
\end{lem}
\begin{proof}[Proof of Lemma \ref{lemI5.2}]
Observe that
\begin{align}
S^{(4)}(\tau)
&=\Pi_M S_L(\frac {x_1}2\tau)
S_N(x_1\tau) \Pi_M S_L(-\frac {x_0-x_1}2 \tau) S_N(-x_0 \tau)
\Pi_M S_L(-\frac {x_0-x_1}2\tau) S_N(x_1 \tau) S_L(\frac{x_1}2 \tau) \notag \\
& = \underbrace{\Pi_M S_L(\frac {x_1}2\tau) S_N(x_1\tau) S_L(-\frac{x_1-\epsilon_1} 2 \tau)}_{=:T_A}
\; \underbrace{\Pi_M S_L(\frac{2x_1-x_0-\epsilon_1} 2 \tau) S_N(-x_0\tau)
S_L (-\frac {x_0-x_1} 4 \tau)}_{=:T_B} \notag \\
& \qquad \underbrace{S_L(-\frac{x_0-x_1} 4 \tau) S_N(x_1\tau) S_L(\frac{x_1}2 \tau)
}_{=:T_C},
\end{align}
where $\epsilon_1=0.01$.  Clearly the operators $T_A$, $T_B$, $T_C$ fulfill the conditions of
Lemma \ref{lemI5.1}. The $H^1$ estimates then follow easily from the computations outlined
 in Remark \ref{remI5.1} with some necessary adjustment of the constants.

To establish \eqref{I5.16}, we note that  for $k\ge 1$,
\begin{align}
\| S_N(\tau) f - f \|_{H^k} \lesssim O(\tau) \| f \|_{H^k},
\end{align}
provided $\tau \|f \|_{H^1}^2 \ll  1$.  We then rewrite $S^{(4)}(\tau)f$ as
\begin{align}
S^{(4)}(\tau) f = S_L( (2x_1-x_0) \tau) f + \tau \tilde f = S_L(\tau) f +\tau \tilde f,
\end{align}
where $\|\tilde f\|_{H^1} \lesssim \| f \|_{H^1}$. One can then bootstrap the higher
regularity from this. We omit further details.
\end{proof}

\begin{rem} \label{remI5.2}
It is not difficult to check that under the assumption of uniform high Sobolev regularity, 
we have (below we assume $f=\Pi_Mf$)
\begin{align}
 S^{(4)}(\tau)  f  &= S_L(\tau)(  f + \tau ( f - \Pi_M( f^3)  ))+O(\tau^2) \notag \\
 & =\underbrace{ (1-\nu \tau \partial_{xx})^{-1} f + \tau (1-\nu \tau \partial_{xx})^{-1}
 (f - \Pi_M (f^3) ) }_{=:\tilde S(\tau) f} +O(\tau^2).
\end{align}
It follows that
\begin{align}
E(S^{(4)}(\tau) f) = E(\tilde S(\tau) f ) +O(\tau^2).
\end{align}
\end{rem}

\begin{lem}[Control of the energy flux] \label{lemI5.3}
Let $\nu>0$ and $M\ge 2$. Suppose $f \in H^2(\mathbb T)$ satisfies $f=\Pi_M f$ and
\begin{align} \label{I5.22}
\| \nu \partial_{xx} f - \Pi_M( f^3) +f\|_2 \le 1.
\end{align}
Then 
\begin{align} \label{5.23}
\| f \|_{H^{80}(\mathbb T)} \le C^{(o)}_{\nu},
\end{align}
where $C^{(o)}_{\nu}>0$ depends only on $\nu$.  Furthermore if 
$\tau \le l_0 M^{-2}$ and $0<\tau \le \tau^{(0)}(\nu, l_0)$
where $\tau^{(0)} (\nu,l_0)>0$ is a sufficiently small constant depending  on ($\nu$, $l_0$),  then
\begin{align} \label{CnuU}
E( S^{(4)}(\tau) f ) \le C^{(U)}_{\nu,l_0},
\end{align}
where $C^{(U)}_{\nu,l_0}>0$ depends only on  ($\nu$, $l_0$).
\end{lem}
\begin{proof}
The estimate \eqref{5.23} follows from energy estimates using \eqref{I5.22}. Note that the
condition $f=\Pi_M f$ is used in the identity $\int \Pi_M(f^3) f dx = \int f^4 dx$. The estimate
\eqref{CnuU} follows from \eqref{5.23} and Lemma \ref{lemI5.1}.
\end{proof}

\begin{lem}[One-step strict energy dissipation with nontrivial energy flux] \label{lemI5.4}
Let $\nu>0$, $M\ge 2$ and $0<\tau \le M^{-2} l_0$. Suppose $f\in H^{80}(\mathbb T)$ with
$f=\Pi_M f$ and
satisfies
\begin{align}
& \| \nu \partial_{xx} f - \Pi_M (f^3)+f \|_{2} \ge 1, \notag \\
& \| f\|_{H^{80}(\mathbb T)} \le B_0<\infty, 
\end{align}
where  $B_0>0$ is a  given constant. There exists $\tau_3=\tau_3(\nu, l_0,
B_0)>0$ sufficiently small such that if $0<\tau \le \tau_3$, then
\begin{align}
E( S^{(4)}(\tau) f) < E(f).
\end{align}
\end{lem}
\begin{proof}[Proof of Lemma \ref{lemI5.4}]
By using Remark \ref{remI5.2}, it suffices for us to show
\begin{align}
E(\tilde S(\tau) f ) + c_1 \tau \le E(f),
\end{align}
for some $c_1>0$.  Denote $w= \tilde S(\tau) f$ and observe that
\begin{align}
\frac {w-f} {\tau} = \nu \partial_{xx} w + f - \Pi_M (f^3).
\end{align}
It follows that for $\tau>0$ sufficiently small, 
\begin{align}
E(w) -E(f) + \frac {c_2}{\tau} \| w -f \|_2^2 \le 0,
\end{align}
where $c_2>0$ is a constant. Now we clearly have
\begin{align}
(\frac 1 {\tau} -\nu \partial_{xx} ) (w-f) = \nu \partial_{xx} f + f - \Pi_M(f^3).
\end{align}
The desired result clearly follows.
\end{proof}

\begin{thm} \label{thmI5.1}
Let $\nu>0$, $M\ge 2$ and $0<\tau\le l_0 M^{-2}$. Assume $a \in H^1(\mathbb T)$ 
 and $\| a \|_{H^1} \le \gamma_1$.  Define $u^0=\Pi_M a$ and
\begin{align}
u^{n+1} = S^{(4)} (\tau) u^n, \quad n\ge 0.
\end{align}
There exists $\tau_*=\tau_*(\nu, l_0, \gamma_1)>0$ sufficiently small such that if
$0<\tau \le \tau_*$, then
\begin{align} 
\sup_{n\ge 1} \| u^n \|_{H^1(\mathbb T)} \le F^{(0)}_{\nu, l_0, \gamma_1},
\end{align}
where $F^{(0)}_{\nu, l_0, \gamma_1} >0$ depends only on ($\nu$, $l_0$, $\gamma_1$).
\end{thm}

\begin{proof}[Proof of Theorem \ref{thmI5.1}]
Denote 
\begin{align}
G= 10(1+\gamma_1+ C_{\nu,l_0}^{(U)} +C_{\nu}^{(o)} ),
\end{align}
where the constants $C_{\nu,l_0}^{(U)}$, $C_{\nu}^{(o)}$ are the same as in 
Lemma \ref{lemI5.3}. In the argument below we shall assume $\tau>0$ is sufficiently small.
The needed smallness (i.e. the existence of $\tau_*$) can be easily worked out by fulfilling the conditions
needed in Lemma \ref{lemI5.2}, Lemma \ref{lemI5.3} and Lemma \ref{lemI5.4}.

By Lemma \ref{lemI5.2}, we can find $c_1=c_1(\nu,l_0, G)>0$ such that 
\begin{align}
\sup_{1\le n \le \frac {c_1}{\tau} } E(u^n) \le  \frac 12 G.
\end{align}

\underline{Claim}: We have
\begin{align}
\sup_{n\ge \frac {c_1} { \tau}} E(u^n) \le  G.
\end{align}

To prove the claim we argue by contradiction.   Suppose $n_0\ge \frac {c_1} {\tau}$ is the 
first integer such that 
\begin{align} \label{I5.35}
E(u^{n_0}) \le G, \quad E(u^{n_0+1}) >G.
\end{align}
By Lemma \ref{lemI5.3}
we must have
\begin{align} \label{I5.36a}
\| \nu \partial_{xx} u^{n_0} - \Pi_M ((u^{n_0})^3 )
+ u^{n_0} \|_2 >1.
\end{align}
Since $n_0\ge \frac {c_1}{\tau}$,  we have $E(u^{n_0- j_0}) \le G$ for some integer
$\frac {c_0} {10\tau} \le j_0 \le \frac {c_0} {10\tau}+2$.  By using smoothing estimates we obtain
\begin{align} \label{I5.36b}
\| u^{n_0} \|_{H^{80}(\mathbb T)} \le C_{\nu, l_0, G},
\end{align}
where $C_{\nu,l_0, G}>0$ depends on ($\nu$, $l_0$, $G$). Since $G$ depends on ($\nu$, $l_0$, $\gamma_1$),
we have $C_{\nu, l_0, G}$ depends only on ($\nu$, $l_0$, $\gamma_1$). 
By \eqref{I5.36a}, \eqref{I5.36b} and Lemma \ref{lemI5.4}, we obtain for sufficiently small
$\tau$ that
\begin{align}
E(u^{n_0+1} ) < E(u^{n_0})
\end{align}
which is clearly a contradiction to \eqref{I5.35}. Thus we have proved the claim.
\end{proof}

\begin{proof}[Proof of Theorem \ref{thm0}]
The $H^1$ estimate follows from Theorem \ref{thmI5.1}. Higher order estimates follow
from the smoothing estimates.
\end{proof}

\frenchspacing
\bibliographystyle{plain}

\end{document}